\documentclass[12pt]{article}

\usepackage{amsmath, amsthm, amssymb}

\def\lexi{<_{\rm lex}}

\def\NZQ{\mathbb}               % the font for N,Z,Q,R,C

\def\ZZ{{\NZQ Z}}
\def\RR{{\NZQ R}}

\def\bold{\mathbf}

\def\eb{{\bold e}}

\def\tb{{\bold t}}

\def\opn#1#2{\def#1{\operatorname{#2}}} % to make operators
\opn\chara{char}
\opn\length{\ell}
\opn\pd{pd}
\opn\rk{rk}
\opn\projdim{proj\,dim}
\opn\injdim{inj\,dim}
\opn\rank{rank}
\opn\depth{depth}
\opn\grade{grade}
\opn\height{height}
\opn\embdim{emb\,dim}
\opn\codim{codim}

\opn\Tr{Tr}
\opn\bigrank{big\,rank}
\opn\superheight{superheight}\opn\lcm{lcm}
\opn\trdeg{tr\,deg}%
\opn\reg{reg}
\opn\lreg{lreg}
\opn\skel{skel}
\opn\com{com}
\opn\div{div}
\opn\Div{Div}
\opn\cl{cl}
\opn\Cl{Cl}
\opn\Spec{Spec}
\opn\Supp{Supp}
\opn\supp{supp}
\opn\Sing{Sing}
\opn\Ass{Ass}
\opn\Ann{Ann}
\opn\Rad{Rad}
\opn\Soc{Soc}
\opn\Ker{Ker}
\opn\Coker{Coker}
\opn\Im{Im}
\opn\Hom{Hom}
\opn\Tor{Tor}
\opn\Ext{Ext}
\opn\End{End}
\opn\Aut{Aut}
\opn\id{id}

\opn\nat{nat}
\opn\pff{proof}%   \pf exists already
\opn\Pf{Pf}
\opn\GL{GL}
\opn\SL{SL}
\opn\mod{mod}
\opn\ord{ord}
\opn\aff{aff}
\opn\con{conv}
\opn\relint{relint}
\opn\st{st}
\opn\lk{lk}
\opn\cn{cn}
\opn\core{core}
\opn\vol{vol}
\opn\link{link}
\opn\star{star}
\opn\gr{gr}

\def\pot#1#2{#1[\kern-0.28ex[#2]\kern-0.28ex]}

\opn\dirlim{\underrightarrow{\lim}}
\opn\inivlim{\underleftarrow{\lim}}
\let\to=\rightarrow

\def\Implies{\ifmmode\Longrightarrow \else
     \unskip${}\Longrightarrow{}$\ignorespaces\fi}
\def\implies{\ifmmode\Rightarrow \else
     \unskip${}\Rightarrow{}$\ignorespaces\fi}
\def\iff{\ifmmode\Longleftrightarrow \else
     \unskip${}\Longleftrightarrow{}$\ignorespaces\fi}
\let\:=\colon
\newtheorem{Theorem}{Theorem}[section]
\newtheorem{Lemma}[Theorem]{Lemma}

\newtheorem{Proposition}[Theorem]{Proposition}

\newtheorem{Example}[Theorem]{Example}

\let\epsilon\varepsilon
\let\phi=\varphi
\let\kappa=\varkappa

\opn\initial{in}
\opn\inim{inm}
\opn\rev{rev}
\opn\Gin{Gin}
\opn\Lex{Lex}
\opn\Shift{Shift}
\opn\shift{shift}
\opn\rate{rate}
\opn\Mon{Mon}
\opn\lex{lex}
\opn\rev{rev}
\opn\red{red}
\opn\max{max}
\opn\min{min}
\opn\initial{in}
\opn\Ker{Ker}
\opn\GL{GL}
\opn\proj{proj}
\begin{document}
\date{}
\title{\bf
Simple polytopes arising from finite graphs
}

\author{Hidefumi Ohsugi and Takayuki Hibi}

\maketitle                        %%%% To set Title and Author names.

%\thispagestyle{empty}

%%%% Replace with your Abstract.

%\begin{multicols}{2}

\begin{abstract}
Let $G$ be a finite graph allowing loops, 
having no multiple edge and no isolated vertex.
We associate $G$
with the edge polytope ${\cal P}_G$
and the toric ideal $I_G$.
By classifying graphs whose edge polytope is simple,
it is proved that  
the toric ideals $I_G$ of $G$
possesses a quadratic Gr\"obner basis
if the edge polytope ${\cal P}_G$ 
of $G$ is simple.
It is also shown that, for a finite graph $G$, 
the edge polytope is simple but not a simplex if and only if 
it is smooth but not a simplex.
Moreover, the Ehrhart polynomial and the normalized volume
of simple edge polytopes are computed.
\end{abstract}

\section*{Introduction}

Let 
$G$
be a finite graph on the vertex set
$V(G) = \{ 1, \ldots, d \} $
allowing loops, 
having no multiple edge and no isolated vertex.
% allowing loops
%and having no multiple edges, 
%and
Let $E(G) = \{ e_1, \ldots, e_n \}$
denote the set of edges (and loops) of $G$.  
If $e = \{ i, j \}$ is an edge
of $G$ between $i \in V(G)$ and $j \in V(G)$, 
then we define $\rho(e) = \eb_i + \eb_j$.
Here $\eb_i$ is the $i$th unit coordinate 
vector of $\RR^d$.
In particular, for a loop 
$e = \{ i, i \}$ at $i \in V(G)$,
one has $\rho(e) = 2 \eb_i$.
The {\em edge polytope} of $G$  
is the convex polytope ${\cal P}_G$
$(\subset \RR^d)$ which is the convex hull of
the finite set 
$\{ \rho(e_1), \ldots, \rho(e_n) \}$.
If $e = \{ i, j \}$ is an edge of $G$,  
then $\rho(e)$ cannot be a vertex of ${\cal P}_G$ 
if and only if $i \neq j$ and $G$ has a loop
at each of the vertices $i$ and $j$. 
With considering this fact, 
throughout the present paper, we assume that
$G$ satisfies the following condition:
\begin{enumerate}
\item[$(*)$]
If $i$, $j \in V(G)$ and 
if $G$ has a loop at each of $i$ and $j$, 
then the edge $\{ i, j \}$ belongs to $G$. 
\end{enumerate}
Let $K[\tb] = K[t_1, \ldots, t_d]$ 
denote the polynomial ring in $d$ variables 
over $K$.
If $e = \{ i, j \}$ is an edge of $G$, then
$\tb^e$ stands for the monomial $t_it_j$
belonging to $K[\tb]$.
Thus in particular, if
$e = \{ i, i \}$ is a loop of $G$ at $i \in V(G)$,
then $\tb^e = t_i^2$.
The {\em edge ring} of $G$ is the affine 
semigroup ring $K[G]$ $(\subset K[\tb])$ 
which is generated by
$\tb^{e_1}, \ldots, \tb^{e_n}$ over $K$.
Let ${\cal R}_G = K[\{x_{i j}\}_{\{i,j\} \in E(G)}]$ denote
the polynomial ring in $n$ variables over $K$.
The {\em toric ideal} of $G$ is the ideal
$I_G$ $(\subset {\cal R}_G)$ which is the kernel
of the surjective ring homomorphism
$\pi : {\cal R}_G \to K[G]$ defined by setting
$\pi(x_{i j}) = t_i t_j$ for $\{i,j\} \in E(G)$.
A convex polytope ${\cal P}$
of dimension $d$ is {\em simple}
if each vertex of ${\cal P}$
belongs to exactly $d$ edges of ${\cal P}$.
A simple polytope ${\cal P}$ is {\em smooth}
if at each vertex of ${\cal P}$,
the primitive edge directions form a lattice basis.

Our goal is as follows:
%\begin{itemize}
%\item %[
(i)
%]
To classify graphs whose edge polytope is simple (Theorem \ref{bunrui}),
%\item %[
(ii)
%]
To show the existence of a quadratic and squarefree initial ideal
of toric ideals arising from simple edge polytopes (Theorem \ref{Boston}),
%\item %[
(iii)
%]
To compute the Ehrhart polynomial and the normalized volume of 
simple edge polytopes (Theorem \ref{normalizedvolume}).
%\end{itemize}

\section{Simple edge polytopes}

In this section, we classify all finite graphs $G$
for which the edge polytope ${\cal P}_G$
is simple.   

A {\em closed walk} of $G$ of length $q$ is a sequence 
$(e_{i_1}, e_{i_2}, \ldots, e_{i_q})$
of edges of $G$, 
where 
$e_{i_k} = \{ u_k, v_k \}$
for $k = 1, \ldots, q$, such that
$v_k = u_{k+1}$ for $k = 1, \ldots, q - 1$ 
together with $v_{q} = u_{1}$.
Such a closed walk is called a {\em cycle}
of length $q$ 
if $u_k \neq u_{k'}$ for all $1 \leq k < k' \leq q$.
In particular, a loop is a cycle of length 1.
Let $\Gamma = (e_{i_1}, e_{i_2}, \ldots, e_{i_{2 q}})$
where 
$e_{i_k} = \{ u_k, v_k \}$
for $k = 1, \ldots, 2 q$ be an even closed walk of $G$.
Then it is easy to see that
$f_{\Gamma} = \prod_{\ell=1}^q x_{u_{2 \ell -1} v_{2 \ell -1}} -  \prod_{\ell=1}^q x_{u_{2 \ell} v_{2 \ell}}$
belongs to $I_G$.
An even closed walk $\Gamma$ of $G$ is called {\it trivial} if $f_\Gamma = 0$.
It is known \cite[Lemma 1.1]{OhHibinomial}
that 
\begin{Proposition}
\label{evenclosedwalk}
Let $G$ be a finite graph.
Then
$\{ f_\Gamma \ | \ \Gamma \mbox{ is a nontrivial even closed walk of } G \}$
is a set of generators of $I_G$.
In particular,
$G$ has no nontrivial even closed walk
if and only if $I_{G} = (0)$.
In addition, if $G$ has at most one loop, then
its edge polytope ${\cal P}_{G}$ is a simplex
if and only if $I_{G} = (0)$.
\end{Proposition}

Let $\widetilde{G}$ denote the subgraph of $G$
with the edge set
$
E(\widetilde{G}) = 
E(G) \setminus 
\{
\{i,j\} \in E(G) \ | \  i \neq j \mbox{ and } \{i,i\} , \{j,j\} \in E(G)
\}.
$
Then we have ${\cal P}_G = {\cal P}_{\widetilde{G}}$.
The following Propositions
are known \cite{normal, Stu}.

\begin{Proposition}
\label{simplexcondition}
Let $G$ be a finite graph.
Then ${\cal P}_G$ is a simplex
if and only if $\widetilde{G}$ satisfies %the following conditions:
both
%\begin{itemize}
%\item[(i)]
(i) each connected component of $\widetilde{G}$ has at most one cycle
and
%\item[(ii)]
(ii) the length of any cycle of $\widetilde{G}$ is odd. 
%\end{itemize}
\end{Proposition}

\begin{Proposition}
Let $G$ be a finite graph and let $G_1, \ldots, G_r$ denote
connected components of $G$.
Then, for each $1 \leq i \leq r$, $\dim {\cal P}_{G_i}$ equals to
$$
\left\{
\begin{array}{cc}
|V(G_i)| - 2 & \ \ \ \mbox{if } G_i \mbox{ has no odd cycle}\\
%\\
|V(G_i)| -1 & \ \ \ \mbox{otherwise}
\end{array}
\right.
$$
and
$\dim {\cal P}_G = r-1+ \sum_{i=1}^r \dim {\cal P}_{G_i} = | V(G)|  -r' - 1$
where $r'$ is the number of connected components of $G$ having no odd cycle.
\end{Proposition}

Let $W$ be a subset of $V(G)$.  The {\em induced
subgraph} of $G$ on $W$ is the subgraph $G_W$ 
on $W$ whose edges are those edges 
$e = \{ i, j \} \in G(E)$
with $i \in W$ and $j \in W$.
If $G'$ is an induced subgraph of $G$
with the edges $e_{i_1}, \ldots, e_{i_q}$, then
we write ${\cal F}_{G'}$ for
the convex hull of 
$\{ \rho(e_{i_1}), \ldots, \rho(e_{i_q}) \}$ 
in $\RR^n$.  It follows that ${\cal F}_{G'}$
is a face of ${\cal P}_G$.
Hence, ${\cal F}_{G'}$ is simple if ${\cal P}_G$
is simple.

\begin{Lemma}
\label{edge}
Let $e$ and $f$ be edges of $G$ with $e \neq f$
and suppose that each of $\rho(e)$ and $\rho(f)$
is a vertex of ${\cal P}_G$.
Then the convex hull of
$\{ \rho(e), \rho(f) \}$ % in $\RR^n$
is an edge of ${\cal P}_G$
if and only if one of the following %conditions 
is satisfied:
\begin{enumerate}
\item[(i)]
Each of $e$ and $f$ is a loop of $G$;
\item[(ii)]
$e = \{ i, j \}$ with $i \neq j$ and
$f$ is a loop at $i$; 
\item[(iii)]
$e = \{ i, j \}$, %with $i \neq j$
$f = \{ k, k \}$ % a loop at $k$ with $k \not\in \{ i, j \}$
with $|\{i,j,k\}|=3$
such that either $\{ i, k \} \not\in E(G)$ 
or $\{ j, k \} \not\in E(G)$;   
\item[(iv)]
$e = \{ i, j \}$, %with $i \neq j$ 
$f = \{ j,k \}$ with %$k \neq \ell$
$|\{i,j,k\}|=3$
such that either $\{ i, k \} \not\in E(G)$ 
or $\{ j, j \} \not\in E(G)$;   
\item[(v)]
$e = \{ i, j \}$ %with $i \neq j$ 
and
$f = \{ k, \ell \}$ with %$k \neq \ell$
$|\{i,j,k,\ell\}|=4$
such that 
the induced subgraph of $G$
on $\{ i, j , k, \ell\}$
has a vertex of degree 1.
{\em (}If the induced subgraph of $G$
on $\{ i, j , k, \ell\}$
has no loop, then we can replace ``a vertex of degree 1"
with ``no even cycle of the form
$(e, e', f, f')$".{\rm )}
\end{enumerate}
\end{Lemma}

\begin{proof}
Let, in general, ${\cal P}$ be a convex polytope
and ${\cal F}$ a face of ${\cal P}$.
Then each face ${\cal F}'$ of ${\cal F}$ is
a face of ${\cal P}$.  In addition, if ${\cal F}$
and ${\cal F}'$ are faces of ${\cal P}$
with ${\cal F}' \subset {\cal F}$,
then ${\cal F}'$ is a face of ${\cal F}$.

(i)  Let $e$ be a loop at $i \in V(G)$ and 
$f$ a loop at $j \in V(G)$, where $i \neq j$.
Let $G'$ be the induced subgraph of $G$
on $\{ i, j \}$.  
By the condition $(*)$, we have $E(G') = \{e,f,\{i,j \} \}$.
Hence 
the convex hull of
$\{ \rho(e), \rho(f) \}$ coincides with
the face ${\cal F}_{G'}$ of ${\cal P}_G$.
Thus the convex hull of
$\{ \rho(e), \rho(f) \}$
is an edge of ${\cal P}_G$.

(ii)  
Let $G'$ be the induced subgraph of $G$
on $\{ i, j \}$.  
Since $\rho(e)$ is a vertex of ${\cal P}_G$,
we have $\{j,j\} \notin E(G)$.
Hence $E(G') = \{ e, f \}$.
As in (i),
the convex hull of
$\{ \rho(e), \rho(f) \}$ coincides with
the face ${\cal F}_{G'}$ of ${\cal P}_G$.

(iii) 
Suppose
$\{ i, k \} \in E(G)$ and 
$\{ j, k \} \in E(G)$.
Let $G'$ be a subgraph of $G$
with $E(G') = \{e,f,\{ i, k \},\{j, k \} \}$.
Then
the edge polytope ${\cal P}_{G'}$ 
is a rectangle and
the convex hull of
$\{ \rho(e), \rho(f) \}$
is its diagonal.
Thus the convex hull of
$\{ \rho(e), \rho(f) \}$ 
cannot be an edge of ${\cal P}_{G'}$.
Since ${\cal P}_{G'}$ is a subpolytope of ${\cal P}_{G}$,
$\{ \rho(e), \rho(f) \}$ 
cannot be an edge of ${\cal P}_{G}$.

On the other hand, 
suppose $\{ i, k \} \not\in E(G)$.
Let $G''$ be the induced subgraph of $G$
on $\{ i, j, k \}$.
By the condition $(*)$,
$E(G'')$ is one of the following:
%\begin{itemize}
%\item
$
\{e,f \}$, 
$\{e,f,\{j,k\} \}$ and 
$
\{e,f,\{j,k\},\{j,j\} \}.
$
%\end{itemize}
In all of three cases above,
the face ${\cal F}_{G''}$ 
is a simplex.
Thus
the convex hull of
$\{ \rho(e), \rho(f) \}$ 
is an edge of ${\cal P}_G$. 

(iv)  
If both $e' = \{ i, k \}$ and $f'= \{ j, j \}$ belongs to $E(G)$, 
then we have
$\rho(e) + \rho(f)
= \rho(e') + \rho(f')$.
Thus the convex hull of
$\{ \rho(e), \rho(f) \}$ 
cannot be an edge of ${\cal P}_G$.

On the other hand, suppose that 
either $\{ i, k \} \not\in E(G)$ 
or $\{ j, j \} \notin E(G)$.
Let $G'$ be the induced subgraph of $G$
on $\{ i, j,k\}$.
If $\{ j, j \} \in E(G)$, then $\{ i, k \} \not\in E(G)$ and
$G$ has a loop at neither $i$ nor $k$.
Hence $E(G') = \{ e,f,\{j,j\}  \}$.
Thus the face ${\cal F}_{G'}$
is a triangle and 
the convex hull of
$\{ \rho(e), \rho(f) \}$ 
is an edge of ${\cal P}_G$. 
If $\{ j, j \} \notin E(G)$, then
the face of ${\cal P}_{G}$
arising from the supporting hyperplane 
${\cal H} =
\{ (x_1, \ldots, x_n) \in \RR^n \ | \ 
x_i + 2 x_j + x_k = 3 \}$
is
the convex hull of
$\{ \rho(e), \rho(f) \}$.

(v)
Let $G'$ be the induced subgraph of $G$
on $\{ i, j ,k, \ell\}$.
Suppose that $G'$ has no vertex of degree 1.
If $G'$ has an even cycle of the form
$(e, e', f, f')$,
then $\rho(e) + \rho(f)
= \rho(e') + \rho(f')$.
Thus the convex hull of
$\{ \rho(e), \rho(f) \}$ 
cannot be an edge of ${\cal P}_G$.
If $G'$ has no even cycle of the form
$(e, e', f, f')$, then we may assume that
$\{i,k\} \notin E(G)$ and $\{i,\ell\} \notin E(G)$.
Since the degree of $i$ is not 1,
we have $\{i,i\} \in E(G)$.
Hence $G$ has a loop at none of $j$, $k$ and $\ell$.
Since the degree of each of $k$ and $\ell$ is not 1,
we have $\{j,k\} \in E(G)$ and $\{j,\ell\} \in E(G)$.
Thus $E(G')=\{ e,f, \{i,i\}, \{j,k\}, \{j,\ell\}\}$.
Then we have
$
2 \rho(e) + \rho(f)
= \rho(\{i,i\}) + \rho(\{j,k\})+ \rho(\{j,\ell\})
$
and hence 
the convex hull of
$\{ \rho(e), \rho(f) \}$ 
cannot be an edge of ${\cal P}_G$.

Suppose that the degree of a vertex $i$ of $G'$ is 1.
Then, none of $\{i,i\}$,
$\{ i, k \}$
and $\{ i, \ell \}$ belongs to $E(G)$.
Since $G$ has no loop at either $k$
or $\ell$, we may assume that
$G$ has no loop at $k$.
Then the hyperplane 
$
{\cal H}' =
\{ (x_1, \ldots, x_n) \in \RR^n \ | \ 
2 x_i + x_j + 2 x_k + x_{\ell} = 3 \}
$
is a supporting hyperplane of ${\cal P}_G$
and the face ${\cal H}' \cap {\cal P}_G$
is either an edge or a triangle.
Since both $\rho(e)$ and $\rho(f)$ belongs to
${\cal H}' \cap {\cal P}_G$,
the convex hull of
$\{ \rho(e), \rho(f) \}$
is an edge of ${\cal P}_G$.
\end{proof}

\begin{Lemma}
\label{newlemma2}
Suppose that ${\cal P}_G$ is simple, but not a simplex.
Then $G$ is a connected graph
having no vertex of degree $1$.
Moreover if an induced subgraph $G'$ of $G$ has no isolated vertex
and ${\cal F}_{G'}$ is not a simplex,
then $G'$ is a connected graph
having no vertex of degree $1$.
\end{Lemma}

\begin{proof}
Suppose that $G$ is not connected.
Since ${\cal P}_G$ is not a simplex,
there exists a connected component $G'$ of $G$ such that 
${\cal P}_{G'}$ is not a simplex.
Then $G'$ has at least $\dim {\cal P}_{G'} +2$ edges.
Let $e = \{i,j\} \in E(G) \setminus E(G')$.
Let $G''$ be the induced subgraph of $G$ on the vertex set
$V(G') \cup \{i,j\}$.
Thanks to Lemma \ref{edge}, the convex hull of 
$\{ \rho(e) , \rho(f) \}$ is an edge of ${\cal P}_{G'}$
for all $f \in E(G')$.
Since $\dim {\cal P}_{G''} = \dim {\cal P}_{G'} + 1$,
the face ${\cal P}_{G''}$ of ${\cal P}_G$ is not simple.
Thus ${\cal P}_G$ is not simple.
Suppose that $G$ has a vertex $i$ of degree $1$.
Let $e = \{ i,j \} \in E(G)$. 
Thanks to Lemma \ref{edge},
the convex hull of 
$\{ \rho(e) , \rho(f) \}$ is an edge of ${\cal P}_{G}$
for all $(e \neq) \ f \in E(G)$.
Since ${\cal P}_G$ is not a simplex,
this contradicts that ${\cal P}_G$ is simple.
Moreover, the induced subgraph $G'$ satisfies the same condition since 
the face ${\cal F}_{G'}$ is also simple.
\end{proof}

\begin{Lemma}
\label{newlemma}
Work with the same notation as above.
%
%Let, as before, $\widetilde{G}$ denote the subgraph of $G$
%with the edge set
%$
%E(\widetilde{G}) = 
%E(G) \setminus 
%\{
%\{i,j\} \in E(G) \ | \  i \neq j \mbox{ and } \{i,i\} , \{j,j\} \in E(G)
%\}.
%$
If ${\cal P}_G$ is simple, but not a simplex,
then $\widetilde{G}$ is a connected graph having no vertex of degree $1$.
\end{Lemma}

\begin{proof}
It follows by the same argument as in Proof of Lemma \ref{newlemma2}
since ${\cal P}_G = {\cal P}_{\widetilde{G}}$.
\end{proof}

A finite graph $G$ on $V(G)$ is called {\em bipartite}
if there is a decomposition
$V(G) = V_1 \cup V_2$ such that
every edge of $G$ is of the form $\{ i, j \}$
with $i \in V_1$ and $j \in V_2$.

\begin{Lemma}
\label{bipartite}
Let $W$ be the set of those vertices 
$i \in V(G)$
such that $G$ has no loop at $i$ and
$G'$ the induced subgraph of $G$ on $W$.
Suppose that 
% the edge polytope ${\cal P}_G$ is simple
% and that 
the face 
${\cal F}_{G'}$ is not a simplex.  
% Either $G'$ has no closed walk of even length or
Then ${\cal P}_G$ is simple
if and only if 
the following conditions are satisfied:
\begin{enumerate}
\item[(i)]
$G'$ is a complete bipartite graph
with at least one cycle of length $4$;
\item[(ii)]
If $G$ has a loop at $i$, then
$\{ i, j \} \in E(G)$ 
for all $j \in W$;
\item[(iii)]
$G$ has at most one loop.
\end{enumerate}
\end{Lemma}

\begin{proof}
{\bf (``Only if'')}
Suppose that ${\cal P}_G$ is simple.

(i) Since ${\cal F}_{G'}$ is a face of ${\cal P}_G$,
it follows that ${\cal F}_{G'}$ is simple.
% Since $G'$ has no loop,
% ${\cal F}_{G'}$ is a simplex
% if and only if $G'$ has no cycle of 
% even length.  
Lemma \ref{edge} (iv) and (v) say that,
if $G'$ has no cycle of length $4$,
then the convex hull of 
$\{ \rho(e), \rho(f) \}$
is an edge of ${\cal F}_{G'}$
for $e, f \in V(G')$ with $e \neq f$.
Since ${\cal F}_{G'}$ is simple,
if $G'$ has no cycle of length $4$,
then ${\cal F}_{G'}$ is a simplex.
Since ${\cal F}_{G'}$ is not a simplex,
it follows that $G'$ has a cycle  
$C$ of length $4$.  Let $G''$ denote 
the induced subgraph of $G$
on $V(C)$.  Since ${\cal F}_{G''}$ is simple,
by using Lemma \ref{edge} (iv), 
it follows easily that $G'' = C$.
In other words, 

\medskip

$(\sharp)$ every cycle $C'$ of length $4$ 
of $G'$ coincides with the induced subgraph on $V(C')$.

\medskip

%We show that $(\sharp)$, 
%for an edge 
%$e = \{ k, \ell \}$ of $G'$,  
%either $k$ or $\ell$ belongs to $V(C)$.
%In fact, 
%if neither $k$ nor $\ell$ belongs to $V(C)$
%and if $G'''$ is the induced subgraph of $G$ 
%on $U \cup \{ k, \ell \}$, then
%the face ${\cal F}_{G'''}$ is of dimension $3$,
%but the vertex $\rho(e)$ of ${\cal F}_{G'''}$
%belongs to $4$ edges of ${\cal F}_{G'''}$.  
%Thus the face ${\cal F}_{G'''}$ of ${\cal P}_G$
%cannot be simple.

Since $C$ is a complete bipartite graph,
there exists a maximal (with respect to inclusion)
complete bipartite subgraph $\Gamma$ of $G'$ with
$C \subset \Gamma$.  
Thanks to $(\sharp)$, $\Gamma$ is the induced subgraph of $G'$ on $V(\Gamma)$.
Let $V(\Gamma) = V_1 \cup V_2$ be the partition
of $V(\Gamma)$. 
Suppose that %%%$\Gamma \neq G'$.  Then
there is an edge $\{ i, j \}$ of $G'$
with $\{ i, j \} \not\in E(\Gamma)$.
%By using $(\sharp)$, 
If $i \notin V(\Gamma)$ and $j \notin V(\Gamma)$,
then by applying Lemma \ref{newlemma2} to the induced subgraph of $G'$ on 
$V(C) \cup \{i,j\}$, there exists an edge $\{k,\ell\} \in E(G') \setminus E(\Gamma)$ 
with $k \in \{i,j\}$ and $\ell \in V(\Gamma)$.
Thus we may assume that one of
$i \in V(\Gamma)$ and $j \in V(\Gamma)$ holds.
Say $i \in V_1$ and $j \not\in V(\Gamma)$.
Let $i' \in V_1$ with $i' \neq i$
and $C'$ a cycle
of length $4$ in $\Gamma$ with $i, i' \in V(C')$,
say $V(C') = \{ i, i', k, k' \}$
with $k, k' \in V_2$ and $k \neq k'$. 
Let $\Gamma'$ denote the induced subgraph
of $G$ on $\{ i, i', k, k', j \}$.
Applying Lemma \ref{newlemma2} to $\Gamma'$,
the degree of $j$ is not 1 in $\Gamma'$.

%\smallskip

%{\bf Case 1.}
If $\{i',j\} \in E(\Gamma')$, then
%
%\noindent
thanks to $(\sharp)$,
$E(\Gamma') = E(C') \cup \{ \{i,j\}, \{i',j\} \}$,
i.e.,
$\Gamma'$ is a complete bipartite graph.
If
%{\bf Case 2.}
$\{i',j\} \notin E(\Gamma')$ and $\{j,k\} \in E(\Gamma')$, then
%\noindent
by Lemma \ref{edge} (iv) and (v), 
the convex hull of $\{\rho(\{i,j\}), \rho(e) \}$ is an edge of ${\cal F}_{\Gamma'}$
for all $e \in E(C') \cup \{ \{j,k\}  \}$.
Hence $\rho(\{i,j\})$ belongs to at least 5 edges of ${\cal F}_{\Gamma'}$.
Since $\dim {\cal F}_{\Gamma'} = 4$, ${\cal F}_{\Gamma'}$ is not simple 
and this is a contradiction.

Thus the induced subgraph $\Gamma''$ of $G$
on $V(\Gamma) \cup \{ j \}$ is a complete
bipartite graph with $\{ i, j \} \in E(\Gamma'')$.
This contradicts the maximality of $\Gamma$.  
Hence $E(\Gamma) = E(G')$.
After the proof of (iii),
we will prove that $V(\Gamma) = V(G') \ (= W)$, that is,
there exists no isolated vertex in $G'$.

(ii) 
We will prove that if $G$ has a loop $e$ at $i$, then
$\{ i, j \} \in E(G)$ 
for all $j \in V(\Gamma)$.
Since $\Gamma$ is a complete bipartite graph
with at least one cycle of length $4$,
it is enough to show that,
for any cycle $C'$ of length $4$ in $\Gamma$,
$\{ i, j \} \in E(G)$ 
for all $j \in V(C')$.
Let $G''$ denote the induced subgraph of $G$
on $V(C') \cup \{ i \}$.
Then ${\cal F}_{G''}$ is simple.

Let $E(C') = \{ \{j,k\},\{j,k'\},\{j',k\},\{j',k'\} \}$.
Let $a$ denote the number of vertices $\ell \in \{ j,j' \}$
with $\{ \ell, i \} \in E(G)$ and
$b$ the number of vertices $\ell' \in \{k,k'\}$
with $\{ \ell', i \} \in E(G)$.
Then by using Lemma \ref{edge} 
the number of edges of ${\cal F}_{G''}$
to which $\rho(e)$ belongs 
is 
$
a + b + 4 - ab 
$.
Since $\dim {\cal F}_{G''} \leq 4 $ and ${\cal F}_{G''}$ is simple,
we have $a + b + 4 - ab \leq 4$, that is, $1 \leq (a-1)(b-1)$.
Since $0 \leq a,b \leq 2$, it follows that either $a= b = 0 $
or $a = b = 2$.
If $a = b = 0$, then $\dim  {\cal F}_{G''}  = 3$ and 
$a + b + 4 - ab = 4 > 3$.
%If either $a < |W_1|$ or $b < |W_2|$,
%then $\rho(e)$ belongs to at least
%$|W_1| + |W_2| + 1$ 
%edges of ${\cal F}_{G''}$.
%Since the dimension of ${\cal F}_{G''}$ is 
%at most $|W_1| + |W_2|$,
Hence the face ${\cal F}_{G''}$ cannot be simple.
Thus $a = b = 2$.%, as required.

(iii) Suppose that $G$ has more than 
one loop.
Let $\{i,i\}, \{i',i'\} \in E(G)$ with $i \neq i'$
and let $C'$ an even cycle of length 4 in $\Gamma$.
Let $G'''$ denote the induced subgraph of $G$ on
$\{i ,i'\} \cup V(C')$.
By using (ii) which we already proved, 
$
E(G''') = \{ \{i,i\}, \{i',i'\}, \{i ,i'\}  \} \cup E(C') \cup 
\{ \{ j, k\} \ | \ j \in \{i,i'\}, k \in V(C')\}
.$
%Let $1, \ldots, q$, where $q \geq 2$,
%be the vertices of $G$ with loops.
%Thus $W = \{ q + 1, \ldots, d \}$ and,
%by using (i) which we already proved, 
%the induced subgraph $G'$ on $W$
%is a complete bipartite graph with
%at least one cycle of length $4$.
%Let $W = W_1 \cup W_2$ and
%each edge of $G'$ is of the form $\{ i, j \}$
%with $i \in W_1$ and $j \in W_2$.
%Fix $e = \{ i_0, j_0 \} \in E(G')$
%with $i_0 \in W_1$ and $j_0 \in W_2$.
%If $f$ is an edge of $G$ of the form
%either $\{ i_0, j \}$ with
%$j \in \{ 1, \ldots, q \} \cup W_2$
%or $\{ i, j_0 \}$ with
%$i \in \{ 1, \ldots, q \} \cup W_1$,
Then Lemma \ref{edge} (iv) and (v) guarantee that
for each $e \in E(C')$,
the vertex $\rho(e)$ belongs to 
6 edges of ${\cal F}_{G'''}$.
%the convex hull of $\{ \rho(e), \rho(f) \}$
%is an edge of ${\cal P}_G$.
%Hence the vertex $\rho(e)$ belongs
%to $2q + (d - q - 2)$ edges of ${\cal P}_G$.
%Since $q \geq 2$, the vertex $\rho(e)$ belongs
%to at least $d$ edges of ${\cal P}_G$.
Since $\dim {\cal F}_{G'''} = 5$,
the face ${\cal F}_{G'''}$ cannot be simple.
This is a contradiction and hence $G$ has at most one loop.

Finally, since
$G$ has at most one loop and has no vertex of degree 1,
$G'$ has no isolated vertex.
Thus $\Gamma = G'$, as desired.

\smallskip
\noindent
{\bf (``If'')}
Suppose that the conditions (i), (ii) and (iii)
are satisfied.  

First, if $G$ has no loop, then $G$ is a complete
bipartite graph on $W = \{ 1, \ldots, d \}$
with at least one cycle of length $4$.
Let $W = W_1 \cup W_2$ be the decomposition of $W$ 
such that each edge of $G$ is of the form
$\{ i, j \}$ with $i \in W_1$ and $j \in W_2$,
where $|W_1| \geq 2$ and $|W_2| \geq 2$.
Fix an edge $e = \{ i_0, j_0 \}$ of $G$.
Let $f = \{ i, j \} \in E(G)$ with $e \neq f$.
Lemma \ref{edge} (iv) and (v) say that
the convex hull of $\{ \rho(e), \rho(f) \}$
is an edge of ${\cal P}_G$ if and only if
either $i = i_0$ or $j = j_0$.  Hence
each vertex of ${\cal P}_G$ belongs to
$d - 2$ edges of ${\cal P}_G$.  
Since ${\cal P}_G$ is of dimension $d - 2$,
it follows that ${\cal P}_G$ is simple.

Second, suppose that $G$ has exactly one loop.
Let $e$ be a loop of $G$ at $d$.
Then each edge $\{ i, d \}$ with $1 \leq i < d$
belongs to $G$. 
 
Fix an edge $e_0 = \{ i_0, j_0 \}$ of $G'$.
Let $f = \{ i, d \}$ with $1 \leq i < d$.
Then by using Lemma \ref{edge} (iv) and (v)
the convex hull of $\{ \rho(e_0), \rho(f) \}$
is an edge of ${\cal P}_G$ if and only if
either $i = i_0$ or $i = j_0$. 
In addition, by using Lemma \ref{edge} (iii)
the convex hull of $\{ \rho(e_0), \rho(e) \}$ 
cannot be an edge of ${\cal P}_G$.
Hence $\rho(e_0)$ belongs to $(d - 3) + 2 = d-1$ edges
of ${\cal P}_G$.

Fix $f_0 = \{ i_0, d \} \in E(G)$ 
with $i_0 \in W_1$.
Let $f \in E(G)$.
Then the convex hull of $\{ \rho(f_0), \rho(f) \}$
is an edge of ${\cal P}_G$ if and only if
$f = \{ i_0, j \}$ with $j \in W_2$, 
or
$f = \{ i, d \}$ with $i_0 \neq i \in W_1$, or
$f = e$. 
Hence $\rho(f_0)$ belongs to $(d - 2) + 1 = d-1$ edges
of ${\cal P}_G$.  

Let $f \in E(G) \setminus \{ e \}$.
Then the convex hull of $\{ \rho(e), \rho(f) \}$
is an edge of ${\cal P}_G$ if and only if
$f = \{ i, d \}$ 
with $1 \leq i < d$.
Hence $\rho(e)$ belongs to $d - 1$ edges
of ${\cal P}_G$.  

Since the dimension of ${\cal P}_G$ is
$d - 1$, it follows that ${\cal P}_G$ is simple,
as desired.
\end{proof}

\begin{Theorem}
\label{bunrui}
Let $W$ denote the set of those vertices 
$i \in V(G)$
such that $G$ has no loop at $i$ and
$G'$ the induced subgraph of $G$ on $W$.
Then the following conditions are equivalent {\rm :}
\begin{itemize}
\item[(i)]
${\cal P}_G$ is simple, but not a simplex {\rm ;}
\item[(ii)]
${\cal P}_G$ is smooth, but not a simplex {\rm ;}
\item[(iii)]
$W \neq \emptyset$ and
$G$ is one of the following graph {\rm :}
\begin{itemize}
\item[($\alpha $)]
$G$ is a complete bipartite graph with at least one cycle
of length $4${\rm ;}
\item[($\beta $)]
$G$ has exactly one loop,
$G'$ is a complete bipartite graph and
if $G$ has a loop at $i$, then
$\{ i, j \} \in E(G)$ 
for all $j \in W$
{\rm ;}
\item[($\gamma $)]
$G$ has at least two loops,
$G'$ has no edge and
if $G$ has a loop at $i$, then
$\{ i, j \} \in E(G)$ 
for all $j \in W$.
\end{itemize}
\end{itemize}
\end{Theorem}

\begin{proof}
%\noindent
%{\it Proof of Theorem \ref{bunrui}.}
{\bf
(``(ii) $\Rightarrow $ (i)")
}
Obvious.

\noindent
{\bf
(``(i) $\Rightarrow $ (iii)")
}
Suppose that 
the edge polytope ${\cal P}_G$ is simple,
but not a simplex.
By Lemma \ref{newlemma2},
$G$ is connected.
Moreover,
by Lemma \ref{newlemma},
we have $W \neq \emptyset$.
%It is known \cite[Corollary 5.4]{OhHibinomial} 
%that, if $G$ has no loops, then
%(i) $\Leftrightarrow $ (iii) ($\alpha$) holds.
%Hence we may assume that $G$ has at least one loop.
If ${\cal F}_{G'}$ is not a simplex, then thanks to 
Lemma \ref{bipartite},
$G$ satisfies either ($\alpha$) or ($\beta$).

Suppose that 
${\cal F}_{G'}$ is a simplex. 
Since ${\cal P}_G$ is not a simplex,
$G$ has at least one loop.
Let $G'_1, \ldots, G'_q$ denote the connected
components of $G'$.
Since ${\cal F}_{G'}$ is a simplex,
it follows that each of the faces ${\cal F}_{G'_1},
\ldots, {\cal F}_{G'_q}$ is a simplex
and that $I_{G'_i} = (0)$ for
$i = 1, \ldots, q$, that is,
each of the connected components 
$G'_1, \ldots, G'_q$ has no nontrivial even closed walk.

Since $G$ is connected and has at least one loop,
for each $G'_r$,
there exists a loop $e = \{i,i\} \in E(G)$ such that
$\{ i, j \}$ with $j \in V(G'_r)$ 
belongs to $E(G)$.
%
%{\bf (a)}
%
%Let $e$ be a loop at $i \in V(G)$.
Suppose that 
%$\{ i, j \}$ with $j \in V(G'_r)$ 
%belongs to $E(G)$ and that
$G'_r$ contains a cycle of
odd length. 
Let $f = \{ j, j' \} \in E(G'_r)$.
Let $G''$ denote the induced subgraph of $G$
on $V(G'_r) \cup \{ i \}$.
Let $a$ denote the dimension of the face
${\cal F}_{G'_r}$.
Then the dimension of ${\cal F}_{G''}$
is $a + 1$.
However, the number of edges of the face 
${\cal F}_{G''}$ to which the vertex $\rho(f)$ 
belongs is $a + 2$.
This contradiction shows that 
%if $\{ i, j \}$ with $j \in V(G'_r)$ 
%belongs to $E(G)$, then
either $E(G'_r) = \emptyset$ or
$G'_r$ is a tree for all $r$.
(A {\em tree} is a connected graph with 
no cycle.)

\smallskip

{\bf Case 1.}
$E(G'_1) = \cdots = E(G'_{q}) = \emptyset$.

%
%{\bf (e)}
%
\noindent
By Lemma \ref{newlemma2},
the degree of each vertex in $W$ is at least 2.
Hence $G$ has at least 2 loops.
Let $e$ be a loop at $i \in V(G)$
and $e'$ a loop at $i' \in V(G)$,
where $i \neq i'$.
Let 
$U_e = \{ j \in W \ | \  \{i,j\} \in E(G) \}$ and 
$U_{e'} = \{ j \in W \ | \  \{i',j\} \in E(G) \}$.
%denote the set of
%vertices $j_r \in V(G'_r)$
%with $E(G'_r) = \emptyset$
%such that $\{ i, j_r \} \in E(G)$.
%Let $U_{e'}$ denote the set of
%vertices $j_r \in V(G'_r)$
%with $E(G'_r) = \emptyset$
%such that $\{ i', j_r \} \in E(G)$.
We show that, 
if $U_{e} \cap U_{e'} \neq \emptyset$,
then $U_{e} = U_{e'}$.
To see why this is true,
let $j_r \in U_{e} \cap U_{e'}$
and $j_s \in U_{e} \setminus U_{e'}$.
Let $G''$ denote the induced subgraph of $G$
on $\{ i, i', j_r, j_s \}$.
Then the dimension of the face ${\cal F}_{G''}$
is $3$.  However, the vertex
$\rho(\{i, j_s\})$ belongs to $4$ edges of
${\cal F}_{G''}$.  Hence
${\cal F}_{G''}$ cannot be simple.
Thus $U_{e} = U_{e'}$.
Thanks to Lemma \ref{newlemma},
it follows that
$U_{e} = W$ holds for any loop $e$ of $G$.
Then $G$ satisfies the condition ($\gamma$).

{\bf Case 2.}
$E(G'_{q}) \neq \emptyset$.

%
%{\bf (c)}
%
\noindent
Let $e$ be a loop at $i \in V(G)$
and $e'$ a loop at $i' \in V(G)$ with $i \neq i'$.
Suppose that $\{ i, j \}$ with $j \in V(G'_r)$
and $\{ i', j' \}$ with $j' \in V(G'_r)$ 
belong to $E(G)$.
Let $f = \{ j, j'' \} \in E(G'_r)$.
Let $G''$ denote the induced subgraph of $G$
on $V(G'_r) \cup \{ i, i' \}$.
Let $a$ denote the dimension of the face
${\cal F}_{G'_r}$.
Then the dimension of ${\cal F}_{G''}$
is $a + 3$.
However, the number of edges of the face 
${\cal F}_{G''}$
to which the vertex $\rho(f)$ belongs is 
at least $a + 4$.  
This contradiction shows that if 
$\{ i, j \}$ with $j \in V(G'_r)$
and $\{ i', j' \}$ with $j' \in V(G'_r)$ 
belongs to $E(G)$, then $E(G'_r) = \emptyset$.
Thus, in particular,
there exists exactly one loop $e_0= \{i_0,i_0\}$ 
of $G$ such that
$\{ i_0, j \}$ with $j \in V(G'_q)$ 
belongs to $E(G)$.

We will show that $G$ has no loop except for $e_0$.
Suppose that $G$ has at least two loops.
Thanks to Lemma \ref{newlemma}, 
all loops of $G$ and $E(G'_{q})$ are
in the connected graph $\widetilde{G}$.
Hence there exists a loop
$e = \{i ,i \}$ of $G$ with $e \neq e_0$ such that
both $\{i_0, j_s\}$ and $\{i , j_s\}$ belongs to $E(G)$
for some $E(G'_s) = \emptyset$ 
with $V(G'_s) = \{ j_s \}$.
%
%{\bf (d)}
%Let $e$ be a loop at $i \in V(G)$.
%Let %$E(G'_r) \neq \emptyset$ and 
%$E(G'_s) = \emptyset$ 
%with $V(G'_s) = \{ j_s \}$.
%Suppose that $\{ i, j \}$ with $j \in V(G'_q)$
%and $\{ i, j_s \}$
%belong to $E(G)$.
%Let $e'$ be a loop at $i' \in V(G)$
%with $i \neq i'$.
%We claim that $\{ i', j_s \} \not\in E(G)$.
Let $f \in E(G'_q)$.
Let $G''$ denote the induced subgraph of $G$
on $V(G'_q) \cup \{ j_s, i_0, i  \}$.
Let $a$ denote the dimension of the face
${\cal F}_{G_q'}$.
%If $\{ i , j_s \} \in E(G)$, then
Then the dimension of ${\cal F}_{G''}$
is $a + 4$.
However, the number of edges of the face 
${\cal F}_{G''}$
to which the vertex $\rho(f)$ belongs is 
at least $a + 5$.  
This is a contradiction.
% shows that 
%$\{ i', j_s \} \not\in E(G)$, as desired.
Thus $G$ has exactly one loop $e_0=\{i_0,i_0\}$.

Since $G$ has no vertex of degree 1 by Lemma \ref{newlemma2},
we have $E(G'_i) \neq \emptyset$ for all $i$.
Next we will show $q = 1$.
%
%{\bf (b)}
%
%Let $e$ be a loop at $i \in V(G)$.
Suppose that $\{ i_0, j \}$ and $\{ i_0, j' \}$
with $j, j' \in V(G'_r)$ and $j \neq j'$
belong to $E(G)$.
Let $f \in E(G'_s)$ with $s \neq r$.
Let $G''$ denote the induced subgraph of $G$
on $V(G'_r) \cup V(G'_s)$
and $G'''$ the induced subgraph of $G$
on $V(G'_r) \cup V(G'_s) \cup \{ i_0 \}$.
Let $a$ denote the dimension of the face
${\cal F}_{G''}$.
Then the dimension of
${\cal F}_{G'''}$ is $a + 2$.
However, 
the number of edges of the face ${\cal F}_{G'''}$
to which the vertex $\rho(f)$ belongs is at least
$a + 3$.  
This contradiction shows that if 
$\{ i_0, j \}, \{ i_0, j' \} \in E(G)$,
where $j, j' \in V(G'_r)$
with $j \neq j'$,
and if there is $s \neq r$ 
with $E(G'_s) \neq \emptyset$, 
then $j = j'$.
Hence, if $q \geq 2$, then at most one vertex of $G'_i$
is incident to the loop $e_0$.
Then there exists a vertex of degree 1
and hence it contradicts to Lemma \ref{newlemma2}.
Thus $q = 1$ and at least two vertices of
$G'_1 = G'$ are incident to the loop $e_0$.

Since $q = 1$, $G'$ is a tree on the vertex set $W$ and
$V(G) = W \cup \{ i_0 \}$.
%
%{\bf (f)}
%Let $E(G'_1) = \cdots = E(G'_{q -1}) = \emptyset$
%and $E(G'_q) \neq \emptyset$.
%Let $e$ be a loop at $i$. 
Let 
$
U = \{ j \in W \ | \  \{ i_0, j \} \in E(G)\}
$.
%$U \subset W$ denote the set of those vertices
%$j \in W$ such that 
%$\{ i_0, j \}$ belongs to $E(G)$.
%Suppose that 
Then, $|U| \geq 2$,
%$$
%\begin{array}{ccccc}
$\dim {\cal P}_G = |V(G)| -1 = |W|$, and
$\dim {\cal P}_{G'}=|W| - 2$.
%\end{array}
%$$
%Let $G''$ denote the induced subgraph of $G$
%on $V(G') \cup \{ i_0 \}$.
%Let $a$ denote the dimension of the face
%${\cal F}_{G'}$.
%Then the dimension of ${\cal F}_{G''}$
%is $a + 2$.
If there is an edge $f = \{ k, \ell \} \in E(G')$
with $k \not\in U$ and 
$\ell \not\in U$, then
the number of edges of 
${\cal P}_{G}$
to which the vertex $\rho(f)$ belongs is 
at least $|W| + 1 $, a contradiction.  
Hence, if $\{ k, \ell \}$ is an edge of $G'$,
then either $k \in U$ or $\ell \in U$.

Suppose that $U \neq W$.
Let $j \in W \setminus U$.
Since $G$ has no vertex of degree 1,
the degree of $j$ in the graph $G'$ is at least 2.
Hence suppose that
$ \{j,i_1\}$ and $\{j,i_2\}$ belong to $E(G')$.
Since $j \notin U$, we have $i_1, i_2 \in U$.
Then the induced subgraph $\widehat{G}$ on 
$\{i_0,j,i_1,i_2\}$ satisfies 
$E(\widehat{G}) =\{ \{i_0,i_0\}, \{j,i_1\}, \{j,i_2\}, \{i_0,i_1\}, \{i_0,i_2\}\}$.
Since the face ${\cal F}_{\widehat{G}}$
is not simple, this is a contradiction.
Thus, we have $U =W$.

Let $\{k, k'\}, \{\ell, \ell'\} \in E(G')$
and suppose that
$\{k, k'\} \cap \{\ell, \ell'\} = \emptyset$.
%Let $k, \ell \in U$. 
Since $G'$ has no cycle of length $4$,
we may assume that $\{k, \ell' \} \notin E(G)$
and $\{k' , \ell' \} \notin E(G)$.
Then the number of edges of
${\cal P}_{G}$
to which the vertex $\rho(\{k, k'\})$ belongs is at least
$|W|+1$, a contradiction.
Hence if $\{k, k'\}, \{\ell, \ell'\} \in E(G)$, 
then $\{k, k'\} \cap \{\ell, \ell'\} \neq \emptyset$.
Since $G'$ is a tree, it follows that
$G'$ has a vertex $j_0$
such that each edge of $G'$ is of the form
$\{ j_0, j \}$. 
%Now, a routine computation shows that 
%Since $G$ has no vertex of degree 1, 
%either $U = W$
%or $U = W \setminus \{ j_0 \}$.
%Since the edge polytope ${\cal P}_{\widehat{G}}$
%of the graph $\widehat{G}$ with 
%the edge set $E(\widehat{G}) =\{ \{i_0,i_0\}, \{j_0,i_1\}, \{j_0,i_2\}, \{i_0,i_1\}, \{i_0,i_2\}\}$
%is not simple, $\widehat{G}$ is not an induced subgraph of $G$.
%Hence $U \neq W \setminus \{ j_0 \}$.
Hence $G'$ is a complete bipartite graph and 
$G$ satisfies the condition $(\beta)$.

\noindent
{\bf
(``(iii) $\Rightarrow $ (ii)")
}
Thanks to Proposition \ref{simplexcondition},
${\cal P}_G$ is not a simplex.
We will show that ${\cal P}_G$ is smooth.
For each vertex $v$ of ${\cal P}_G$,
let $M_v$ denote the matrix whose columns are
set of all primitive edge directions at $v$.
Let ${\bf 1} = (1,\ldots,1)$ and let $E_m$ be the identity matrix.

($\alpha$)
Suppose that $G$ is a complete bipartite graph.
Let $e = \{1,1'\} \in E(G)$.
For $f \in E(G)$, the convex hull of $\{\rho(e), \rho(f)\}$ is an edge of ${\cal P}_G$
if and only if either $f = \{1, i'\}$ with $i' \neq 1'$ or $f = \{i , 1'\}$ with $i \neq 1$.
Then 
%$$
%M_{\rho(e)}
%=
%\left(
%\begin{array}{ccc|ccc}
%-1 & \cdots & -1 &  &  &  \\
%1 &  &  &  &  &  \\
% & \ddots &  &  &  &  \\
% &  & 1 &  &  &  \\
%\hline
% &  &  & -1 & \cdots & -1 \\
% &  &  & 1 &  &  \\
% &  &  &  & \ddots &  \\
% &  &  &  &  & 1 \\
%\end{array}
%\right)
%.$$
$
M_{\rho(e)}
=
\left(
\begin{array}{c|c}
- {\bf 1}  & \\
E_m & \\
\hline
 & -{\bf 1}\\
 & E_{m'}
\end{array}
\right)
.
$
Since they form a lattice basis, ${\cal P}_G$ is smooth.

($\beta$)
Suppose that $G$ has exactly one loop
and $G'$ is a complete bipartite graph with vertex set $V_1 \cup V_2$.
Let $\{1,1\} \in E(G)$,
$V_1 = \{2, \ldots, m\} $ and $V_2 = \{m+1,\ldots,d\}$.

Let $e = \{1,1\}$.
For $f =\{i,j\} \in E(G)$, the convex hull of $\{\rho(e), \rho(f)\}$ is an edge of ${\cal P}_G$
if and only if $i = 1$.
Hence
%$$
%M_{\rho(e)}
%=
%\left(
%\begin{array}{ccc}
%-1 & \cdots & -1\\
%\hline
%1 &  &  \\
% & \ddots &  \\
% &  & 1 
%\end{array}
%\right)
%$$
$
M_{\rho(e)}
=
\left(
\begin{array}{c}
-{\bf 1}\\
E_{d-1}
\end{array}
\right)
$
and columns of $M_{\rho(e)}$ form a lattice basis.

Let $e = \{1,2\} \in E(G)$.
For $f \in E(G)$ with $e \neq f$, 
the convex hull of $\{\rho(e), \rho(f)\}$ is an edge of ${\cal P}_G$
if and only if $f$ belongs to %is one of the following:
%\begin{itemize}
%\item
$
\{ \{1,1\} \} \cup \{ \{1,j\}  \ |  \ j \in V_1 \setminus \{2\} \}
\cup \{ \{2,j\} \ | \ j \in V_2 \}.
$
%\begin{center}
%$f = \{1,1\},
%f = \{1,j\}$ with $j \in V_1 \setminus \{2\}$,
%$f = \{2,j\}$ with $j \in V_2$.
%\end{center}
Hence
%$$
%M_{\rho(e)}
%=
%\left(
%\begin{array}{c|ccc|ccc}
%1 & 0 & \cdots & 0 & -1 & \cdots & -1 \\
%\hline
%-1 &-1 & \cdots & -1 & 0 & \cdots & 0 \\
% & 1 &  &  &  &  &  \\
% &  & \ddots &  &  &  &  \\
% &  &  & 1 &  &  &  \\
%\hline
% &  &  &  & 1 &  &  \\
% &  &  &  &  & \ddots &  \\
% &  &  &  &  &  & 1 \\
%\end{array}
%\right)
%$$
$
M_{\rho(e)}
=
\left(
\begin{array}{c|c|c}
1 & &-{\bf 1}\\
\hline
-1 &-{\bf 1} &  \\
 & E_{m-2} & \\
\hline
 &  & E_{d-m}\\
\end{array}
\right)
$
and columns of $M_{\rho(e)}$ form a lattice basis.

Let $e = \{2, m+1\} \in E(G)$.
For $f =\{i,j\}\in E(G)$ with $e \neq f$, 
the convex hull of $\{\rho(e), \rho(f)\}$ is an edge of ${\cal P}_G$
if and only if
either $i = 2$ or $i = m+1$.
Hence
%$$
%M_{\rho(e)}
%=
%\left(
%\begin{array}{cccc|cccc}
%1 & 0 & \cdots & 0 & 1 & 0 & \cdots & 0 \\
%\hline
% &  &  &  & -1 & -1 & \cdots & -1 \\
% &  &  &  &  & 1 &  &  \\
% &  &  &  &  &  & \ddots &  \\
% &  &  &  &  &  &  & 1 \\
%\hline
%-1 & -1 & \cdots & -1 &  &  &  &  \\
% & 1 &  &  &  &  &  &  \\
% &  & \ddots &  &  &  &  &  \\
% &  &  & 1 &  &  &  &  
%\end{array}
%\right)
%$$
$
M_{\rho(e)}
=
\left(
\begin{array}{c|c|c|c}
1 & & 1 & \\
\hline
 & & -1 & -{\bf 1}\\
 & & & E_{m-2}\\
 \hline
 -1 & -{\bf 1} &  & \\
 & E_{d-m} & & 
\end{array}
\right)
$
and columns of $M_{\rho(e)}$ form a lattice basis.

Thus, in all cases,
columns of $M_{\rho(e)}$ form a lattice basis
and hence ${\cal P}_G$ is smooth.

($\gamma$)
Let $V(G') = \{1,\ldots,m\}$ and $V(G) = \{1,\ldots, m , m+1,\ldots,d\}$.

Let $e = \{1,1\}$.
For $f  \in E(G)$ with $e \neq f$, 
the convex hull of $\{\rho(e), \rho(f)\}$ is an edge of ${\cal P}_G$
if and only if 
either $f = \{i,i\}$ with $2 \leq i \leq m$ or
$f = \{1,j\}$ with $m+1 \leq j \leq d$.
Hence
$
M_{\rho(e)}
=
\left(
\begin{array}{c}
-{\bf 1}\\
E_{d-1}
\end{array}
\right)
$
and columns of $M_{\rho(e)}$ form a lattice basis.

Let $e = \{1,m+1\}$.
For $f  \in E(G)$ with $e \neq f$, 
the convex hull of $\{\rho(e), \rho(f)\}$ is an edge of ${\cal P}_G$
if and only if $f$ belongs to %is one of the following:
%\begin{itemize}
%\item
$\{ \{1,1\} \} \cup
\{ \{i,m+1\} \ | \ 1 < i \leq m\} \cup
\{ \{1,j\} \ | \ m+1 < j \leq d\}
$.
%\end{itemize}
Hence
%$$
%M_{\rho(e)}
%=
%\left(
%\begin{array}{c|ccc|ccc}
%1 &-1 & \cdots & -1 & 0 & \cdots & 0 \\
% & 1 &  &  &  &  &  \\
% &  & \ddots &  &  &  &  \\
% &  &  & 1 &  &  &  \\
%\hline
%-1 & 0 & \cdots & 0 & -1 & \cdots & -1 \\
% &  &  &  & 1 &  &  \\
% &  &  &  &  & \ddots &  \\
% &  &  &  &  &  & 1 \\
%\end{array}
%\right)
%$$
$
M_{\rho(e)}
=
\left(
\begin{array}{c|c|c}
1 & -{\bf 1} & \\
 & E_{m-1} & \\
\hline
 -1 &  & -{\bf 1}\\
 & & E_{d-m-1}
\end{array}
\right)
$
and columns of $M_{\rho(e)}$ form a lattice basis.

Thus, in all cases,
columns of $M_{\rho(e)}$ form a lattice basis
and hence ${\cal P}_G$ is smooth, as desired.
%\qed
\end{proof}

%\smallskip

%Theorem \ref{bunrui}
%enables us to classify the
%finite graphs $G$ for which
%the edge polytope ${\cal P}_G$ is simple
%and 
%the face 
%${\cal F}_{G'}$ is a simplex,
%where $G'$ the induced subgraph of $G$ on $W$
%which is the set of those vertices 
%$i \in V(G)$
%such that $G$ has no loop at $i$. 

%\begin{Example}
%\label{figure}
%{\em
%The edge polytope of each of the finite graphs 
%drawn below is simple.
%}
%\end{Example}

On the other hand,
there exist graphs $G$ such that ${\cal P}_G$ is a simplex,
but not smooth.

\begin{Example}
\label{nonsmooth}
{\rm
The edge polytope of the graphs $G_1$ and $G_2$
is a simplex (triangle) but not smooth:
\begin{itemize}
\item[(i)]
%Let $G_1$ be a graph such that its set of edges is
%$
%E(G_1) = \{ e_1 = \{1,1\}, e_2 = \{1,2\}, e_3 = \{2,2\}, e_4 = \{1,3\}  \}
%.$
$
E(G_1) = \{ \{1,1\}, \{1,2\}, \{2,2\}, \{1,3\}  \}
.$
%The edge polytope ${\cal P}_{G_1}$ is a triangle with the vertices $\rho(e_1)$, $\rho(e_3)$ and $\rho(e_4) $. 
%Since $\rho(e_2) -  \rho(e_4) =  ( 1,  0,   -1 )^T$
%is not spaned by
%the columns of
%$M_{\rho(e_4)}=
%\left(
%\begin{array}{cc}
%2  & 0\\
%-1 & 1\\
%-1 &-1
%\end{array}
%\right)
%,$
%${\cal P}_{G_1}$ is not smooth.
\item[(ii)]
%Let $G_2$ be a graph such that its set of edges is
%$
%E(G_2) = \{ e_1 = \{1,1\}, e_2 = \{1,2\}, e_3 = \{2,2\}, e_4 = \{3,4\}  \}
%.$ 
$
E(G_2) = \{ \{1,1\}, \{1,2\}, \{2,2\}, \{3,4\}  \}
.$ 
%Then ${\cal P}_{G_2}$ is a triangle
%with the vertices $\rho(e_1)$, $\rho(e_3)$ and $\rho(e_4) $. 
%Since $\rho(e_2) -  \rho(e_4) = ( 1, 1,- 1, -1)^T$
%is not spaned by
%the columns of
%$M_{\rho(e_4)}=
%\left(
%\begin{array}{cc}
%2  & 0\\
%0 & 2\\
%-1 & -1\\
%-1 &-1
%\end{array}
%\right)
%,$
%${\cal P}_{G_2}$ is not smooth.
\end{itemize}
}
\end{Example}

\begin{Theorem}
\label{simplexsmooth}
Let $G$ be a finite graph on the vertex set 
$V(G) =\{1,\ldots,d\}$ with $d \geq 2$.
Suppose that ${\cal P}_G$ is a simplex.
Then ${\cal P}_G$ is smooth if and only if 
$G$ satisfies one of the followings:
\begin{itemize}
\item[(i)]
$E(G) = \{ \{i,j\} \ | \ 1 \leq i \leq j \leq d\}$;
\item[(ii)]
$G$ satisfies the condition in Proposition \ref{simplexcondition}
and has at most one loop.
\end{itemize}
\end{Theorem}

\begin{proof}
{\bf (``Only if'')}
Suppose that the edge polytope ${\cal P}_G$
of a graph $G$ is a simplex and smooth.
Let $G^{(1)},\ldots,G^{(r)}$ be the connected components of $G$.  

Suppose that
$G$ has at least two loops
and that $G^{(1)}$ has a loop of $G$.
Thanks to the condition $(*)$, all loops of $G$ are in $G^{(1)}$.
%By virtue of Proposition \ref{simplexcondition},
%we have the following for each $2 \leq i \leq r$:
%\begin{itemize}
%\item
%$G^{(i)}$ has at most one cycle;
%\item
%the length of any cycle of $G^{(i)}$ is odd. 
%\end{itemize}
Since ${\cal P}_G$ is smooth, any face of ${\cal P}_G$ is 
smooth.
Hence none of the graphs $G_1$, $G_2$ in Example \ref{nonsmooth}
is the induced subgraph of $G$.
If $r \geq 2$, then a loop of $G^{(1)}$
together with an edge of $G^{(2)}$
form an induced subgraph of $G$.
This contradicts that 
$G_2$ is not an induced subgraph of $G$.
Hence $r =1$, that is, $G$ is connected.
Suppose that there exists a vertex $i$ such that $G$ has no loop at $i$.
Since $G$ is connected, there exists a path from $i$ to 
a vertex $j$ where $G$ has a loop at $j$.
Then it follows that $G_1$ is the induced subgraph of $G$.
This is a contradiction.
Thus $G$ has a loop at every vertex.

\smallskip

\noindent
{\bf (``If'')}
Suppose that $G$ is a graph satisfying the condition (i).
Since $G$ is the induced subgraph of a graph
in Theorem \ref{bunrui} (iii) ($\gamma$),
${\cal P}_G$ is a face of a smooth polytope.
Hence ${\cal P}_G$ is smooth.

Suppose that $G$ is a graph satisfying the condition (ii).
Since $G$ has at most one loop, 
${\cal P}_G \cap \ZZ^d$ coinsides with the set of vertices of ${\cal P}_G$.
It is well-known that, if $\sigma$ is an integral simplex having no lattice point
except for its vertices, then we can transform $\sigma$ to
the standard simplex of $\RR^{\dim \sigma}$, that is,
the convex hull of $\{{\bf 0}, \eb_1,\ldots,\eb_{\dim \sigma} \}$.
Since the standard simplex is smooth, 
${\cal P}_G$ is smooth.
\end{proof}

\section{Quadratic Gr\"obner bases}% of simple edge polytopes}

%We are now in the position to
In this section, we discuss quadratic Gr\"obner bases of
toric ideals arising from simple edge polytopes.

Let 
${\cal R}_G = K[\{x_{i j}\}_{\{i,j\} \in E(G)}]$ be
the polynomial ring in $n$ variables over $K$ and
the {\em toric ideal} of $G$ is the ideal
$I_G$ $(\subset {\cal R}_G)$ which is the kernel
of the surjective ring homomorphism
$\pi : {\cal R}_G \to K[G]$ defined by setting
$\pi(x_{i j}) = t_i t_j$ for $\{i,j\} \in E(G)$.
Let, as before, $W$ denote the set of those vertices 
$i \in V(G)$
such that $G$ has no loop at $i$ and
$G'$ the induced subgraph of $G$ on $W$.

Suppose that 
the edge polytope ${\cal P}_G$ is simple
and $I_G \neq ( 0 )$.
%, but not a simplex.
Then, either
$G$ satisfies one of $(\alpha)$, $(\beta)$ and $(\gamma)$
in the condition (iii) of Theorem \ref{bunrui}
or 
$G$ has at least two loops and satisfies the condition in Proposition \ref{simplexcondition}.
If $G$ satisfies $(\alpha)$,
then let $V(G) = V_1 \cup V_2 = \{1,\ldots,m\} \cup \{m+1,\ldots,d \}$.
If $G$ satisfies $(\beta)$,
then suppose that $G$ has a loop at $1$ and
let $V(G') = V_1 \cup V_2 = \{2,\ldots,m \} \cup \{m +1,\ldots,d \}$.
If $G$ satisfies $(\gamma)$,
then suppose that $G$ has a loop at each of $\{1,\ldots,m\}$ and
let $V(G') = \{m+1,\ldots,d \}$.
If ${\cal P}_G$ is a simplex with $I_G \neq ( 0 )$, then 
suppose that $G$ has a loop at each of $\{1,\ldots,m\}$.
Let $\lexi$ be a lexicographic order on ${\cal R}_G$
%induced by the ordering
%\begin{itemize}
%\item
satisfying that $x_{k \ell} \lexi x_{ ii }$ where $k < \ell$
and that
$x_{i j} \lexi x_{k \ell}$ where $i < j , k < \ell
\Leftrightarrow$ either (i) $i < k$ or (ii)
$i=k$
and
$j > \ell$.

\begin{Theorem}
\label{Boston}
Work with the same notation as above.
If the edge polytope ${\cal P}_G$ is simple
and $I_G \neq ( 0 )$,
then the reduced Gr\"obner basis of the toric ideal $I_G$
with respect to $\lexi$
consists of quadratic binomials
with squarefree initial monomials.
\end{Theorem}

\begin{proof}
Suppose that ${\cal P}_G$ is not a simplex.
Let $G_d$ be a graph with the edge set
$E(G_d) = \{ \{i,j\} \ | \ 1 \leq i \leq j \leq d \}$.
Then, $K[G_d]$ is called the second Veronese subring
of $K[\tb]= K[t_1,\ldots,t_d]$.
It is known \cite[Chapter 14]{Stu} that,
with respect to a ``sorting" monomial order,
the toric ideal
$I_{G_d}$ has a quadratic Gr\"obner basis
$
{\cal G}_d = 
\{
\underline{x_{i j} x_{k \ell}} - x_{i k} x_{j \ell} %\ (\neq 0)
\ | \ 
i \leq j < k \leq \ell
\}
\cup
\{
\underline{x_{i \ell} x_{j k}} - x_{i k} x_{j \ell} %\ (\neq 0)
\ | \ 
i < j \leq k < \ell
\}
$
where 
the initial monomial of each $g \in {\cal G}_d$
is the first monomial of $g$ and squarefree.
See also \cite{multi}.
Since 
the initial monomial of $g \in {\cal G}_d$ with respect to $\lexi$ is 
also the first monomial, ${\cal G}_d$ is a Gr\"obner basis with respect to
$\lexi$.
It then follows that
if either $x_{i j} x_{k \ell}$ or $x_{i \ell} x_{j k}$ belongs to ${\cal R}_G$,
then $x_{i k} x_{j \ell}$ belongs to ${\cal R}_G$.
By the elimination property of $\lexi$,
${\cal G}_d \cap {\cal R}_G$
is a Gr\"obner basis of $I_G$ with respect to $\lexi$.

Suppose that ${\cal P}_G$ is a simplex with $I_G \neq ( 0 )$.
Let $G^{(1)},\ldots,G^{(r)}$ be the connected components of $G$.  
Let $G^{(1)}$ have a loop of $G$.
Then thanks to the condition $(*)$, all loops of $G$ are in $G^{(1)}$.
By virtue of the condition $(*)$ together with Propositions
\ref{evenclosedwalk} and \ref{simplexcondition},
we have
$
{\cal R}_G  / I_G = \left( {\cal R}_{G_m}  / I_{G_m} \right) [\{x_{i j}\}_{\{i,j\} \in E(G) \setminus E(G_m)}]
$
and hence
$I_G$, $I_{G^{(1)}}$ and $I_{G_m}$ have the same minimal set of binomial generators.
%the following conditions are equivalent
%for the set of binomials of $I_G$:
%\begin{itemize}
%\item
%a minimal set of binomial generators of $I_G$;
%\item
%a minimal set of binomial generators of $I_{G^{(1)}}$;
%\item
%a minimal set of binomial generators of $I_{G_m}$.
%\end{itemize}
Thus ${\cal G}_m$ is a Gr\"obner basis of $I_G$.
\end{proof}

\section{Ehrhart polynomials}

If ${\cal P} \subset \RR^d$ is an integral convex polytope,
then we define $i({\cal P},m)$ by
$
i({\cal P},m) = |m {\cal P} \cap \ZZ^d|
.$
It is known that $i({\cal P},m)$ is a polynomial in $m$ of degree $\dim {\cal P}$.
We call $i({\cal P},m)$ the {\it Ehrhart polynomial} of ${\cal P}$.
If ${\rm vol}({\cal P})$ is the normalized volume of ${\cal P}$,
then
the leading coefficient of $i({\cal P},m)$ is ${\rm vol}({\cal P})/ (\dim {\cal P})!$.
We refer the reader to \cite{Hibi} for the detailed information about
Ehrhart polynomials of convex polytopes.

\begin{Theorem}
\label{normalizedvolume}
Let $G$ be a graph in Theorem \ref{Boston} (iii).
Let $W$ denote the set of those vertices 
$i \in V(G)$
such that $G$ has no loop at $i$ and
$G'$ the induced subgraph of $G$ on $W$.
Then the Ehrhart polynomial $i({\cal P}_G,m)$
and the normalized volume ${\rm vol}({\cal P}_G)$
of the edge polytope ${\cal P}_G$
are as follows;
\begin{itemize}
\item[($\alpha $)]
If $G$ is the complete bipartite graph 
on the vertex set 
$V_1 \cup V_2$ with $|V_1| = p$ and $|V_2| = q$,
then we have
$
i({\cal P}_G,m) = { p + m -1 \choose p-1}{ q + m -1 \choose q-1}
$
and
$
{\rm vol} ({\cal P}_G) = { p + q -2 \choose p-1};
$
\item[($\beta $)]
If $G'$ is the complete bipartite graph 
on the vertex set 
$V_1 \cup V_2$ with $|V_1| = p$ and $|V_2| = q$,
then we have
$
i({\cal P}_G,m) = { p + m \choose p}{ q + m  \choose q}
$
and
$
{\rm vol} ({\cal P}_G) = { p + q \choose p};
$
\item[($\gamma $)]
If $G$ possesses $p$ loops and $V(G) = d$, then we have
$
i({\cal P}_G,m) = 
%{ d + 2 m - 1\choose d-1 } -  
\sum_{j=1}^p 
{ j+m-2 \choose j-1}
{ d-j+m \choose d-j}
$
and
$
{\rm vol} ({\cal P}_G) 
= 
%2^{d-1} - 
\sum_{j=1}^p {d-1 \choose j-1}
.$
\end{itemize}
\end{Theorem}

\begin{proof}
Since $I_G$ possesses a squarefree initial ideal,
$K[G]$ is normal.
(It is known  \cite[Corollary 2.3]{normal} that $K[G]$ is normal 
if and only if $G$ satisfies the ``odd cycle condition.")
Hence, the Ehrhart polynomial $i({\cal P}_G,m)$
coincides with the Hilbert function $\dim_k (K[G])_m $ of the homogeneous
$K$-algebra $K[G] = \bigoplus_{m=0}^\infty (K[G])_m $.
The squarefree quadratic initial ideal of Theorem \ref{Boston}
guarantees that the set of monomials $x_{i_1 j_1} x_{i_2 j_2} \cdots x_{i_m j_m} $
with each $\{i_r, j_r\} \in E(G)$ such that
%\begin{eqnarray}
$1 \leq i_1 \leq  \cdots \leq i_m \leq j_1   \leq \cdots \leq j_m \leq d$
%\end{eqnarray}
is a $K$-basis of $(K[G])_m$.

($\alpha $)
See, e.g.,  \cite[Corollary 2.7 (b)]{multi}.
($\beta $)
%Thanks to Proposition \ref{doukei}, it follows from ($\alpha $).
Let $V(G') = \{1,\ldots,p\} \cup \{p+2,\ldots,d\}$ and $G$ has a loop at $p+1$.
Then $\{i_r ,j_r\} \in E(G)$ if and only if 
$i_r \leq p+1$ and $j_r \geq p+1$.
Hence, the number of sequence above is 
$
i({\cal P}_G,m) = { p + m \choose p}{ q + m  \choose q}
$.
Thus, the leading coefficient of $(d-1)!  i({\cal P}_G,m) $ is
${ p + q \choose p}$.
($\gamma $)
Let $V(G) \setminus W = \{1,2,\ldots,p\}$.
Then $\{i_r ,j_r\} \notin E(G)$ if and only if 
$i_r,j_r > p$.
Since
the number of sequence above with $i_m = j$ is
$
%\sum_{j=1}^p 
{ j+m-2 \choose j-1}
{ d-j+m \choose d-j}
,$
we have the required formula for $i({\cal P}_G,m)$.
Hence, the leading coefficient of $(d-1)!  i({\cal P}_G,m) $ is $ \sum_{j=1}^p {d-1 \choose j-1}$.
\end{proof}

\bigskip

\noindent
Hidefumi Ohsugi\\
Department of Mathematics,
College of Science,
Rikkyo University.\\
{\tt ohsugi@rkmath.rikkyo.ac.jp}

\bigskip

\noindent
Takayuki Hibi\\
Graduate School of Information Science and Technology,
Osaka University.\\
{\tt hibi@math.sci.osaka-u.ac.jp}

%\end{multicols}
\end{document}